\documentclass[10pt,reqno]{amsart}
\usepackage[colorlinks=true,allcolors=blue,backref=page]{hyperref}
\usepackage{color}
\usepackage{amsmath, amssymb, amsthm}

\usepackage{mathrsfs}
\usepackage{mathtools}
\usepackage[noabbrev,capitalize,nameinlink]{cleveref}
\crefname{equation}{}{}
\usepackage{fullpage}
\usepackage[noadjust]{cite}
\usepackage{graphics}
\usepackage{pifont}
\usepackage{tikz}
\usepackage{bbm}
\usepackage[T1]{fontenc}

\usetikzlibrary{arrows.meta}

\usepackage{environ}
\usepackage{framed}
\usepackage{url}
\usepackage[linesnumbered,ruled,vlined]{algorithm2e}
\usepackage[noend]{algpseudocode}
\usepackage[labelfont=bf]{caption}
\usepackage{cite}
\usepackage{framed}
\usepackage[framemethod=tikz]{mdframed}
\usepackage{appendix}
\usepackage{graphicx}
\usepackage[textsize=tiny]{todonotes}
\usepackage{tcolorbox}
\usepackage{enumerate}
\usepackage[shortlabels]{enumitem}
\allowdisplaybreaks[1]

\apptocmd{\sloppy}{\hbadness 10000\relax}{}{} 

\crefname{algocf}{Algorithm}{Algorithms}

\crefname{equation}{}{} 
\crefname{conjecture}{Conjecture}{Conjectures} 
\AtBeginEnvironment{appendices}{\crefalias{section}{appendix}} 

\crefformat{enumi}{#2#1#3}
\crefrangeformat{enumi}{#3#1#4 to~#5#2#6}
\crefmultiformat{enumi}{#2#1#3}%
{ and~#2#1#3}{, #2#1#3}{ and~#2#1#3}

\usepackage[color,final]{showkeys} 

\colorlet{refkey}{orange!20}
\colorlet{labelkey}{blue!30}

\crefname{algocf}{Algorithm}{Algorithms}

\usepackage{aliascnt}

\numberwithin{equation}{section}
\newtheorem{theorem}{Theorem}[section]
\crefname{theorem}{Theorem}{Theorems}
\Crefname{theorem}{Theorem}{Theorems}

\makeatletter
\newcommand{\newthmwithalias}[3]{%
  \newaliascnt{#1}{theorem}%
  \newtheorem{#1}[#1]{#2}%
  \aliascntresetthe{#1}%
  \expandafter\crefname\expandafter{#1}{#2}{#2s}%
  \expandafter\Crefname\expandafter{#1}{#2}{#2s}%
  \expandafter\def\csname #1autorefname\endcsname{#2}%
}
\makeatother

\newthmwithalias{proposition}{Proposition}{}
\newthmwithalias{lemma}{Lemma}{}
\newthmwithalias{claim}{Claim}{}
\newthmwithalias{corollary}{Corollary}{}
\newthmwithalias{conjecture}{Conjecture}{}
\newthmwithalias{fact}{Fact}{}
\newthmwithalias{definition}{Definition}{}
\newthmwithalias{problem}{Problem}{}
\newthmwithalias{question}{Question}{}
\newthmwithalias{assumption}{Assumption}{}
\newthmwithalias{example}{Example}{}
\newthmwithalias{setup}{Setup}{}

\theoremstyle{remark}
\newtheorem*{remark}{Remark}

\newtheorem*{question*}{Question}
\newtheorem*{definition*}{Definition}

\crefname{section}{Section}{Sections}
\Crefname{section}{Section}{Sections}
\crefname{subsection}{Section}{Sections}
\Crefname{subsection}{Section}{Sections}
\crefname{equation}{}{} 
\crefname{algocf}{Algorithm}{Algorithms}
\Crefname{algocf}{Algorithm}{Algorithms}

\newcommand{\mb}{\mathbb}
\newcommand{\mbf}{\mathbf}

\newcommand{\mc}{\mathcal}

\newcommand{\on}{\operatorname}

\newcommand{\imod}[1]{~\mathrm{mod}~#1}

\newcommand{\hide}[1]{}

\allowdisplaybreaks

\title{On random matrices with large corank}

\author[A1]{Zach Hunter}
\address{Department of Mathematics, ETH Z\"{u}rich, Z\"{u}rich, Switzerland.}
\email{zach.hunter@math.ethz.ch}

\author[A2]{Matthew Kwan}
\address{Institute of Science and Technology Austria (ISTA). Am Campus 1, 3400 Klosterneuburg, Austria}
\email{matthew.kwan@ist.ac.at}

\author[A3]{Lisa Sauermann}
\address{Institute for Applied Mathematics, University of Bonn, Germany}
\email{sauermann@iam.uni-bonn.de}

\author[A4]{Mehtaab Sawhney}
\address{Department of Mathematics, Columbia University, New York, NY 10027}
\email{m.sawhney@columbia.edu}

\begin{document}

\begin{abstract} 
Let $1\le k\le n$ and $M$ be a random $n\times n$ matrix with independent uniformly random $\{\pm 1\}$-entries. We show that there exists an absolute constant $c > 0$ such that 
\[\mbf{P}[\on{rank}(M)\le n-k]\le \exp(-c nk).\]
\end{abstract}

\maketitle

\section{Introduction}\label{sec:introduction}
The main result of this paper is the following bound for the probability that a random $\{\pm 1\}$-matrix has large corank. 
\begin{theorem}\label{thm:main}
There exists an absolute constant $c > 0$ such that the following holds. Take $1\le k\le n$, and let $M$ be a random $n\times n$ matrix with independent entries uniformly random in $\{\pm 1\}$. We have 
\[\mbf{P}[\on{rank}(M)\le n-k]\le \exp(-c nk).\]
\end{theorem}

For the sake of simplicity we have restricted to the case of $\{\pm 1\}$ entries, but using the techniques of Bourgain, Vu and Wood \cite{BVW10}, it is plausible that our proof can be adapted to allow the entries of $M$ to come from more flexible entry distributions. 

\subsection{History}
The singularity probability of random $\{\pm 1\}$-matrices has been intensively studied. After a series of works of Koml\'{o}s \cite{Kom67}, Kahn, Koml\'{o}s and Szemer\'{e}di \cite{KKS95}, Tao and Vu \cite{TV06,TV07}, and Bourgain, Vu and Wood \cite{BVW10}, celebrated work of Tikhomirov \cite{Tik20} established that 
\[\mbf{P}[\on{rank}(M)< n]\le \Big(\frac{1}{2} + o(1)\Big)^{n}.\]
The base of $1/2+o(1)$ is seen to be sharp by considering the probability that two rows match.

By considering the probability that $k+1$ rows match, it is natural to conjecture that in fact for all $k$ we have
\[\mbf{P}[\on{rank}(M)\le n-k] = \Big(\frac{1}{2}+o(1)\Big)^{nk}.\]
When $k$ is constant, this was verified by Jain, Sah and the final author \cite{JSS22} (building heavily on aforementioned work of Tikhomirov \cite{Tik20}). However, the methods in that work are not suitable when $k$ grows faster than say $\log n$. More recently, Rudelson~\cite{Rud24} managed to prove \cref{thm:main} for $k\le \sqrt{n}$, via an ingenious geometric argument.

\subsection{Proof ideas}As the proof is rather short, we give only brief comments. The previously-mentioned works \cite{Tik20}, \cite{JSS22} and \cite{Rud24} (along with a whole host of other works) fall broadly within the so-called ``geometric'' approach of Rudelson and Vershynin \cite{RV08}. Approaching \cref{thm:main} in this manner appears to provide substantial resistance. Therefore, our work essentially reverts to the previous strategy of Kahn, Koml\'{o}s and Szemer\'{e}di~\cite{KKS95}. In fact, in some ways our proof is even simpler than in \cite{KKS95}: as we may assume $k$ is larger than an absolute constant, we can avoid certain technicalities about ``structured subspaces'' which makes our argument more streamlined. 

The key new ingredient is a ``high-dimensional relative anticoncentration inequality'' (see \cref{prop:relative}), which provides a comparison between the probability that a random $\{\pm 1\}^{n}$-vector lies in a $(n-k)$-dimensional subspace, and the probability that a suitably ``lazy'' random vector lies in the same subspace (here ``lazy'' means that each component has some probability of being zero instead of $\pm1$). The key technical innovation is to win a factor exponential in $k$ (i.e., a factor of $\gamma^k$ for an absolute constant $\gamma<1$) when comparing the probabilities of these two events. We remark that earlier work of Kahn, Koml\'os and Szemer\'edi~\cite{KKS95} proceeds via a similar inequality for the case $k = 1$ (which in turns builds on techniques of Hal\`{a}sz \cite{Hal75}). Actually, this earlier inequality in some ways is quantitatively stronger than ours (it wins an arbitrarily large constant factor as the lazy random vector becomes more lazy), but it can only be effectively applied to subspaces which are in a certain sense ``unstructured''. A key simplifying feature of our proof is that we do not need to worry about such ``unstructuredness'' properties, and can work with completely arbitrary subspaces of codimension $k$.

The proof of \cref{prop:relative} is closely modeled on the work of Tao and Vu \cite{TV06}, however at a crucial juncture we need to replace an inequality\footnote{This inequality asserts that for any subsets $A,B\subseteq \mbf{T}$ of the circle $\mbf T=\mbf R/\mbf Z$, writing $\mu(\cdot)$ for the  Lebesgue measure, we have $\mu(A+B)\ge \min(1,\mu(A) + \mb(B))$. This inequality, and generalizations thereof, are also known under various other names in the literature (in particular, two well-known generalizations were made by Macbeath~\cite{Mac53} and Kneser~\cite{Kne56}).} due to Raikov \cite{Rai39} 
(which may be viewed as a continuous version of the Cauchy--Davenport inequality) with a suitable ``high-dimensional'' variant. This is achieved via a combination of pigeonholing and compression techniques to reduce to a case where the Brunn--Minkowski inequality in $\mbf{R}^d$ may be invoked. 
\subsection{Acknowledgements}
This work began when the authors were visiting Mathematisches Forschungsinstitut Oberwolfach, which provided ideal working conditions. MS thanks Vishesh Jain for initial discussions regarding the problem.

ZH was supported by SNSF grant 200021-228014. MK was supported by ERC Starting Grant ``RANDSTRUCT'' No.\ 101076777. LS was funded by the Deutsche Forschungsgemeinschaft (DFG,
German Research Foundation) -- CRC 1720 -- 539309657. This research was conducted during the period MS served as a Clay Research Fellow.

\section{Reduction to high-dimensional relative concentration result}
Our key technical ingredient is the following relative concentration result.
\begin{proposition}\label{prop:relative}
There exist absolute constants $\gamma = \gamma_{\ref{prop:relative}}<1$ and $k_{\ref{prop:relative}}\ge 1$ such that the following holds. Let $p \in (0,1/32]$, and let $\mu_{p}$ be the distribution which takes on values $1$ and $-1$ each with probability $p$, and the value $0$ with probability\footnote{Tao and Vu \cite{TV06} and Bourgain, Vu and Wood \cite{BVW10} use the slightly differing convention that $\mu_p$ is $0$ with probability $1-p$ and $\pm 1$ each with probability $p/2$.} $1-2p$. For some $k\ge k_{\ref{prop:relative}}$, let $V\subseteq \mbf{Q}^n$ be a linear subspace of dimension $\dim(V) = n-k$. Then for random vectors $X\sim \mu_{1/2}^{\otimes n}$ and $Y\sim \mu_{p}^{\otimes n}$, we have
\[\sup_{t\in \mbf{Q}^n}\mbf{P}[X + t\in V] \le \gamma^{k} \cdot \mbf{P}[Y\in V].\]
\end{proposition}
\begin{remark}
Note that $\mu_{1/2}$ is the uniform distribution on $\{\pm 1\}$. We only need the case where $t = \vec{0}$ in our proof of \cref{thm:main}. Also, we note that the condition $k\ge k_{\ref{prop:relative}}$ can likely be removed at the cost of slightly complicating the proof.

It would be of interest to understand the quantitative behaviour of $\gamma$ in \cref{prop:relative} when taking $p\to 0$. By considering $V = \{\vec{0}\}\subseteq \mbf{Q}^n$ we see that $\gamma\ge 1/2$ for every $p$; it appears plausible that $\gamma \to 1/2$ as $p\to 0$.
\end{remark}

We postpone the proof of \cref{prop:relative} to the next section and first proceed with the proof of \cref{thm:main}. The proof is closely modeled after the proof of Kahn, Koml\'os and Szemer\'edi \cite{KKS95} (following the presentation of Tao and Vu \cite{TV06}).

\begin{proof}[Proof of \cref{thm:main}]
We may assume that $k$ is sufficiently large (larger than any absolute constant), recalling that $\mbf{P}[\on{rank}(M) < n]\le \exp(-c'n)$ for an absolute constant $c'>0$ by the work of \cite{KKS95}.

We fix $p = 1/32$, and let $\gamma<1$ be an absolute constant as in \cref{prop:relative}. Now, fix $c>0$ such that $e^{-4c}>\max(1-2p,\gamma)$.

Let $\mc{V}$ denote the set of $(n-k)$-dimensional linear subspaces $V\subseteq \mbf{Q}^n$. Consider independent random vectors $X_1,\dots,X_n\sim \mu_{1/2}^{\otimes n}$, and note that (since the columns of $M$ also have the same distribution $\mu_{1/2}^{\otimes n}$)
\[\mbf{P}[\on{rank}(M)= n-k]=\mbf{P}[\on{span}(X_1,\ldots,X_n)\in  \mc{V}].\]
For any $V\in \mc{V}$, we define
\[\rho_V = \mbf{P}_{X\sim \mu_{1/2}^{\otimes n}}[X\in V].\]

Furthermore, define a subspace $V\in \mc{V}$ to be \emph{thin} if $\rho_V\le (1-p)^{n/2}$ and to be \emph{thick} if $\rho_V> (1-p)^{n/2}$. Denoting by $\mc{V}_{\on{thin}}$ and $\mc{V}_{\on{thick}}$ the sets of thin and thick subspaces, respectively, we obtain a partition $\mc{V} = \mc{V}_{\on{thin}} \sqcup \mc{V}_{\on{thick}}$. We bound the probability that $\on{span}(X_1,\ldots,X_n)$ belongs to $\mc{V}_{\on{thin}}$ and to $\mc{V}_{\on{thick}}$, respectively:

\begin{claim}\label{clm:thin}
We have that 
\[\mbf{P}[\on{span}(X_1,\ldots,X_n)\in  \mc{V}_{\on{thin}}]\le 2^n \cdot (1-p)^{nk/2}.\]
\end{claim}
\begin{claim}\label{clm:thick}
We have that 
\[\mbf{P}[\on{span}(X_1,\ldots,X_{n}) \in \mc{V}_{\on{thick}}] \le n2^{2n} \cdot\gamma^{nk/2}.\]
\end{claim}

Combining these two claims, we obtain
\begin{align*}\mbf{P}[\on{rank}(M)= n-k]&=\mbf{P}[\on{span}(X_1,\ldots,X_n)\in  \mc{V}_{\on{thin}}]+\mbf{P}[\on{span}(X_1,\ldots,X_n)\in  \mc{V}_{\on{thick}}]\\
&\le 2^n \cdot(1-2p)^{nk/2}+n2^{2n} \cdot\gamma^{nk/2}\le n2^{2n+1}. \exp(-2cnk)
\end{align*}
Summing this for all $k'=k,k+1,\dots,n$, we can conclude that (using that $k$ is sufficiently large with respect to $c$)
\[\mbf{P}[\on{rank}(M)\le n-k]= \sum_{k'=k}^{n}\mbf{P}[\on{rank}(M)= n-k']\le n^2 2^{2n+1} \cdot\exp(-2cnk)\le \exp(-cnk)\qedhere\]
\end{proof}

It remains to prove the two claims. We start with \cref{clm:thin} handling thin subspaces. 
\begin{proof}[{Proof of \cref{clm:thin}}]
Note that whenever $\on{span}(X_1,\ldots,X_n)=V$ for some $V\in \mc{V}_{\on{thin}}$, we can find a subset $S\subseteq [n]$ of size $|S|=n-k$ such that the vectors $X_s$ for $s\in S$ form a basis of $V$. In other words, we have $\on{span}((X_s)_{s\in S})=V$ and $X_i\in V$ for all $i\in [n]\setminus S$. By symmetry, the probability for this happening is the same for all subsets $S\subseteq [n]$ of size $|S|=n-k$, and so we can conclude
\begin{align*}
\mbf{P}[\on{span}(X_1,\ldots,X_n)\in  \mc{V}_{\on{thin}}]&=\sum_{V\in \mc{V}_{\on{thin}}}\mbf{P}[\on{span}(X_1,\ldots,X_n)=V]\\
&\le \binom{n}{n-k}\sum_{V\in \mc{V}_{\on{thin}}}\mbf{P}[\on{span}(X_1,\ldots,X_{n-k})=V \text{ and }X_{n-k+1},\dots,X_n\in V]\\
&\le 2^n\sum_{V\in \mc{V}_{\on{thin}}}\mbf{P}[\on{span}(X_1,\ldots,X_{n-k})=V] \cdot \rho_V^k\\
&\le 2^n\cdot (1-p)^{nk/2}\sum_{V\in \mc{V}_{\on{thin}}}\mbf{P}[\on{span}(X_1,\ldots,X_{n-k})=V]\le 2^n\cdot (1-p)^{nk/2}.\qedhere
\end{align*}
\end{proof}

We now prove  \cref{clm:thick} about thick subspaces using \cref{prop:relative}.
\begin{proof}[{Proof of \cref{clm:thick}}]
Let $m = \lceil n/2\rceil$ and consider independent random vectors $Y_1,\ldots, Y_m \sim \mu_{p}^{\otimes n}$ (also independent from $X_1,\dots,X_n$). For any $V\in \mc{V}_{\on{thick}}$ and any $i\in [m]$, by \cref{prop:relative} we have that
\[\mbf{P}[Y_i\in V]\ge \gamma^{-k} \rho_V.\]
Therefore we obtain
\begin{align}
\mbf{P}[\on{span}(X_1,\ldots, X_n,Y_1,\ldots, Y_m) = V] &\ge \mbf{P}[\on{span}(X_1,\ldots, X_n) = V \text{ and } Y_1,\ldots, Y_m\in V]\notag \\
&\ge \gamma^{-km}\rho_V^{m}\cdot \mbf{P}[\on{span}(X_1,\ldots, X_n) = V]\label{eq:eq-for-claim-thick}
\end{align}
for any $V\in \mc{V}_{\on{thick}}$. Note that whenever $\on{span}(X_1,\ldots, X_n,Y_1,\ldots, Y_m) = V$, there exist subsets $J\subseteq [m]$ and $I\subseteq [n]$ with $|I|+|J| = n-k$ such that $\on{span}((X_i)_{i\in I}, (Y_j)_{j\in J}) = V$ and $Y_s\in \on{span}((Y_j)_{j\in J})$ for all $s\in [m]\setminus J$  (and $X_s\in V$ for all $s\in [n]\setminus I$). Indeed, one can take $(Y_j)_{j\in J}$ to be a basis of $\on{span}(Y_1,\ldots, Y_m)$ and then extend to a basis of $\on{span}(X_1,\ldots, X_n,Y_1,\ldots, Y_m)=V$. Therefore, using symmetry, we conclude that 
\begin{align*}
\mbf{P}&[\on{span}(X_1,\ldots, X_n,Y_1,\ldots, Y_m) = V]\\
&\le \sum_{r=0}^{m}\binom{m}{r}\binom{n}{n-k-r} \cdot \mbf{P}[\on{span}(Y_1,\ldots,Y_r,X_1,\ldots,X_{n-k-r}) = V \text{ and } Y_{r+1},\dots,Y_m\in \on{span}(Y_1,\ldots,Y_r)\\
&\qquad\qquad\qquad\qquad\qquad\qquad\qquad\qquad\qquad\qquad\qquad\qquad\qquad\qquad\qquad\qquad\qquad\text{ and }X_{n-k-r+1},\dots,X_n\in V]\\
&\le \sum_{r=0}^{m}2^{m}2^n \cdot \mbf{P}[\on{span}(Y_1,\ldots,Y_r,X_1,\ldots,X_{n-k-r}) = V]\cdot (1-p)^{(n-r)(m-r)} \cdot \rho_V^{k+r}. \end{align*}
Here, we used that for any given outcomes of $Y_1,\dots,Y_r$ we have $\mbf{P}[Y_i\in \on{span}(Y_1,\ldots,Y_r)\mid Y_1,\dots,Y_r]\le (1-p)^{n-r}$ for each $i=r+1,\dots,m$, by the weighted Odlyzko lemma (see \cite[Lemma~4.3]{TV07}). Indeed, this lemma shows that $\mbf{P}_{Y\sim \mu_{p}^{\otimes n}}[Y\in W]\le (1-2p)^{n-\dim W}\le (1-p)^{n-\dim W}$ for any linear subspace $W\subseteq \mbf{Q}^{n}$.

Furthermore, observing that $(1-p)^{(n-r)(m-r)}\le \rho_V^{m-r}$ for any $r=0,\dots,m$ and any $V\in \mc{V}_{\on{thick}}$ (for $r=m$ this holds trivially, and for $r\le m-1\le n/2$ we can observe that  $(1-p)^{n-r}\le (1-p)^{n/2}\le \rho_V$), we obtain
\begin{align*}
\mbf{P}[\on{span}(X_1,\ldots, X_n,Y_1,\ldots, Y_m) = V]&\le \sum_{r=0}^{m}2^{2n} \mbf{P}[\on{span}(Y_1,\ldots,Y_r,X_1,\ldots,X_{n-k-r}) = V]\cdot \rho_V^{m-r} \cdot \rho_V^{k+r} \\
&\le 2^{2n}\cdot \rho_V^{m} \cdot\sum_{r=0}^{m}\mbf{P}[\on{span}(Y_1,\ldots,Y_r,X_1,\ldots,X_{n-k-r}) = V].
\end{align*}
Combining this with \cref{eq:eq-for-claim-thick}, we can conclude
\[\mbf{P}[\on{span}(X_1,\ldots, X_n) = V]\le 2^{2n}\cdot \gamma^{km} \cdot\sum_{r=0}^{m}\mbf{P}[\on{span}(Y_1,\ldots,Y_r,X_1,\ldots,X_{n-k-r}) = V]\]
for every $V\in \mc{V}_{\on{thick}}$. Summing this for all $V\in \mc{V}_{\on{thick}}$ gives
\begin{align*}
\mbf{P}[\on{span}(X_1,\ldots, X_n) \in \mc{V}_{\on{thick}}]&=\sum_{V\in \mc{V}_{\on{thick}}}\mbf{P}[\on{span}(X_1,\ldots, X_n) = V]\\
&\le 2^{2n}\cdot \gamma^{km} \cdot\sum_{r=0}^{m}\sum_{V\in \mc{V}_{\on{thick}}}\mbf{P}[\on{span}(Y_1,\ldots,Y_r,X_1,\ldots,X_{n-k-r}) = V]\\
&\le 2^{2n}\cdot \gamma^{km} \cdot\sum_{r=0}^{m} 1\le n2^{2n}\cdot \gamma^{km}\le n2^{2n}\cdot \gamma^{nk/2}.\qedhere
\end{align*}
\end{proof}

\section{Proof of \texorpdfstring{\cref{prop:relative}}{}}
It remains to prove \cref{prop:relative}. The proof is modeled after the Fourier comparison argument in \cite{TV06}; the crucial trick is replacing a certain ``doubling'' inequality of Raikov \cite{Rai39} (which morally is the continuous analogue of the Cauchy--Davenport theorem, a \textit{one-dimensional} result about additive doubling) with a suitable ``high-dimensional'' variant. More precisely, we will use the following ``doubling'' property for subsets of the torus $\mbf{T}^{k}:= (\mbf{R}/\mbf{Z})^k$ satisfying an appropriate coordinate restriction. Throughout the remainder of the paper, we will use $\mu(\cdot)$ to denote Lebesgue measure on $\mbf{T}$, $\mbf{T}^d$ and $\mbf{R}^d$.

\begin{lemma}\label{lem:frieman-check}
Let $A_1,\dots,A_k\subseteq\mbf{T}$ be closed subsets with $\mu(A_i)\le 1/2$ for $i=1,\dots,k$. Then for any closed set $S\subseteq A_1\times \cdots \times A_k\subseteq \mbf{T}^{k}$, we have
\[\mu(S+S)\ge 2^{k} \cdot \mu(S).\]
\end{lemma}
In the case where $A_1,\dots,A_k=[0,1/2]$, the sumset $S+S$ has no ``wraparound'' and therefore \cref{lem:frieman-check} follows from the Brunn--Minkowski inequality (see e.g.\ \cite[Theorem~3.16]{TVbook}). The general result (proven in the next section) reduces to this case via compressions.

We are now ready to prove \cref{prop:relative}.

\begin{proof}[Proof of \cref{prop:relative}]
As in the statement of the proposition, let $p\in (0,1/32]$ and consider an $(n-k)$-dimensional linear subspace $V\subseteq \mbf{Q}^n$. Let $L$ be a $k\times n$ matrix whose rows form a basis of the orthogonal complement of $V$. By permuting columns (which does not change the probability of the underlying event), we may assume that the first $k\times k$ block of $L$ is nonsingular. By performing row operations we may assume that this $k$ by $k$ block is diagonal. By rescaling the rows of $L$ we may assume that all the entries of $L$ are integral and furthermore that the diagonal entries are equal to some integer $Z$ (but note that we have no control over the size of $Z$). Finally we let $w_1,\ldots,w_n$ denote the columns of $L$.

Recall that $\mbf{T}^k$ is the $k$-dimensional torus $(\mbf{R}/\mbf{Z})^k$. For a random vector $X\sim \mu_{1/2}^{\otimes n}$ we have (by the Fourier inversion formula)
\begin{align*}
\on{sup}_{t\in \mbf{Q}^n}\mbf{P}[X+t\in V]&=\on{sup}_{t'\in \mbf{Z}^k}\mbf{P}[LX = t']\\
&=\on{sup}_{t'\in \mbf{Z}^k}\int_{\mbf{T}^k}\mbf{E}[\exp(2\pi i\theta^{T}(LX-t'))]\,d\theta\\
&\le \int_{\mbf{T}^k}\Big|\mbf{E}[\exp(2\pi i\theta^{T}LX)]\Big|\,d\theta\\
&= \int_{\mbf{T}^k}\prod_{j=1}^{n}\Big|\frac{1}{2}\exp(2\pi i\theta^{T}w_j)+\frac{1}{2}\exp(-2\pi i\theta^{T}w_j)\Big|\,d\theta\\
&= \int_{\mbf{T}^k}\prod_{j=1}^{n}|\cos(2\pi\theta^{T}w_j)|\,d\theta\\
&= \int_{\mbf{T}^k}\prod_{j=1}^{n}|\cos(\pi\theta^{T}w_j)|\,d\theta.
\end{align*}
In the final line, we have applied the change of variable $\theta \to \theta/2$ and noted that $|\cos(\theta+\pi)| = |\cos(\theta)|$ to rewrite the integral.

Via a similar computation, for a random vector $Y\sim \mu_{p}^{\otimes n}$ we have
\begin{align*}
\mbf{P}[Y\in V] =\mbf{P}[LY = 0]= \int_{\mbf{T}^k}\mbf{E}[\exp(2\pi i\theta^{T}LY)]\,d\theta =\int_{\mbf{T}^k}\prod_{j=1}^{n}(1-2p + 2p\cos(2\pi\theta^{T}w_j))\,d\theta.
\end{align*}
Note that $1-2p + 2p\cos(\varphi)\ge 1-4p\ge 0$ for all $\varphi\in \mbf{R}$, so all factors in this integral are nonnegative everywhere. Thus, it suffices to show that
\begin{equation}\label{eq:to-show-Fourier}
    \int_{\mbf{T}^k}\prod_{j=1}^{n}|\cos(\pi\theta^{T}w_j)|\,d\theta\le \gamma^k\cdot \int_{\mbf{T}^k}\prod_{j=1}^{n}(1-2p + 2p\cos(2\pi\theta^{T}w_j))\,d\theta
\end{equation}
for some absolute constant $\gamma<1$. To this end, we use the elementary trigonometric inequality (which is essentially \cite[Lemma~7.1~(19)]{TV06}), whose proof can be found at the end of this section.

\begin{lemma}\label{lem:trig-iden}For $p\in (0,1/32]$, and any $\varphi,\varphi'\in \mbf R$, we have
\[|\cos(\varphi)|\cdot |\cos(\varphi')|\le (1-2p + 2p\cos(2\varphi+2\varphi'))^{2}.\]
\end{lemma}

Let us now fix a small absolute constant $\beta>0$ (small enough to satisfy certain inequalities later in the proof), and define $\tau=e^{-\beta k}$. Note that we can rewrite the left hand side of \cref{eq:to-show-Fourier} as
\[\int_{\mbf{T}^k}\prod_{j=1}^{n}|\cos(\pi\theta^{T}w_j)|\,d\theta=\int_{\mbf{T}^k}\min\Big(\prod_{j=1}^{n}|\cos(\pi\theta^{T}w_j)|,\tau\Big)\,d\theta+ \int_{\mbf{T}^k}\max\Big(\prod_{j=1}^{n}|\cos(\pi\theta^{T}w_j)|-\tau,0\Big)\,d\theta.\]

To bound the first summand, note that \cref{lem:trig-iden} (applied, for $j=1,\dots,n$, to $\varphi=\pi\theta^{T}w_j$ and $\varphi'=0$) yields
\[\prod_{j=1}^{n}|\cos(\pi\theta^{T}w_j)|\le \prod_{j=1}^{n}(1-2p + 2p\cos(2\pi\theta^{T}w_j))^2\]
for all $\theta\in \mbf{T}^k$. Therefore, we obtain the bound
\[\int_{\mbf{T}^k}\!\min\Big(\prod_{j=1}^{n}|\cos(\pi\theta^{T}w_j)|,\tau\Big)\,d\theta\le \int_{\mbf{T}^k}\!\tau^{1/2}\prod_{j=1}^{n}|\cos(\pi \theta^{T}w_j)|^{1/2}\,d\theta\le \tau^{1/2}\!\int_{\mbf{T}^k}\prod_{j=1}^{n}(1-2p + 2p\cos(2\pi\theta^{T}w_j))\,d\theta\]
for the first summand. To handle the second summand, we define
\[S_{\eta}:=\Big\{\theta\in \mbf{T}^k : \prod_{j=1}^{n}|\cos(\pi\theta^{T}w_j)|\ge  \eta\Big\}\]
for any $\eta\in [\tau,1]$. Then we have
\[\int_{\mbf{T}^k}\max\Big(\prod_{j=1}^{n}|\cos(\pi\theta^{T}w_j)|-\tau,0\Big)\,d\theta=\int_\tau^1 \mu(S_{\eta})\,d\eta.\]

Recall that the first $k$ columns of $L$ form the $k\times k$ matrix  $ZI_k$ (i.e., a diagonal matrix with $Z$ everywhere on the diagonal). Therefore each $\theta\in S_\eta\subseteq \mbf{T}^k$ satisfies
\[
\prod_{i=1}^{k}|\cos(\pi Z\theta_i)| =\prod_{j=1}^{k}|\cos(\pi\theta^{T}w_j)|\ge \prod_{j=1}^{n}|\cos(\pi\theta^{T}w_j)|\ge \eta \ge \tau=e^{-\beta k}.\]
Note that whenever $Z\theta_i\notin \mbf{Z} + [-1/4,1/4]$, we have $|\cos(\pi Z\theta_i)|\le 2^{-1/2}$. Therefore, for any $\theta\in S_\eta$, there are at most $M:=\lfloor 2\beta k/\log(2)\rfloor$ coordinates $i\in [k]$ with $Z\theta_i\notin \mbf{Z} + [-1/4,1/4]$. For each subset $I\subseteq [k]$ of size $|I|\le M$, let $B_I\subseteq \mbf{T}^k$ denote the ``box'' of all points $\theta\in \mbf{T}^k$ with $Z\theta_i\notin \mbf{Z} + [-1/4,1/4]$ for all indices $i\in I$ and $Z\theta_i\in \mbf{Z} + [-1/4,1/4]$ for all indices $i\in [k]\setminus I$. Then for each $\eta\in [\tau,1]$, we have $S_\eta\subseteq \bigcup_I B_I$, where the union is taken over all subsets $I\subseteq [k]$ of size $|I|\le M$. The number of such subsets $I$ is
\[\sum_{j=0}^M\binom{k}{j}\le k \cdot (ek/M)^{M}\le (3/2)^{k},\]
provided that the absolute constant $\beta>0$ was chosen to be sufficiently small in the beginning of the proof (and provided that $k$ is larger than a suitable absolute constant). Thus, by the pigeonhole principle, there exists a subset $I\subseteq [k]$ such that $\mu(S_{\eta}\cap B_I)\ge (2/3)^{k} \mu(S_{\eta})$. Now, the set $S_{\eta}\cap B_I$ satisfies the assumption of \cref{lem:frieman-check}, and applying the lemma we obtain
 \[\mu(S_{\eta}+S_{\eta})\ge \mu((S_{\eta}\cap B_I) + (S_{\eta}\cap B_I))\ge 2^{k} \mu(S_{\eta}\cap B_I)\ge (4/3)^{k} \cdot \mu(S_{\eta}).\]

Now, for any $\theta,\theta'\in S_{\eta}$, by \cref{lem:trig-iden}, we have
\[\prod_{j=1}^{n}(1-2p+ 2p\cos(2\pi(\theta+\theta')^{T}w_j))\ge \prod_{j=1}^{n} |\cos(\pi\theta^{T}w_j)|^{1/2}\cdot  |\cos(\pi\theta'^{T}w_j)|^{1/2}\ge \eta^{1/2}\cdot \eta^{1/2}= \eta.\]
Thus, we can conclude that
\[\mu(S_\eta)\le (3/4)^k \cdot\mu(S_\eta+S_\eta)\le (3/4)^k\cdot \mu\Big(\Big\{\theta\in \mbf{T}^k : \prod_{j=1}^{n}(1-2p + 2p\cos(2\pi\theta^{T}w_j))\ge  \eta\Big\}\Big)\]
for all $\eta\in [\tau,1]$. Integrating this over the interval $[\tau,1]$ yields 
\[\int_{\mbf{T}^k}\max\Big(\prod_{j=1}^{n}|\cos(\pi\theta^{T}w_j)|-\tau,0\Big)\,d\theta=\int_\tau^1 \mu(S_{\eta})\,d\eta\le (3/4)^k\cdot \int_{\mbf{T}^k} \prod_{j=1}^{n}(1-2p + 2p\cos(2\pi\theta^{T}w_j))\,d\theta.\]

All in all, we can conclude
\[\int_{\mbf{T}^k} \prod_{j=1}^{n}|\cos(\pi \theta^{T}w_j)|\,d\theta \le (\tau^{1/2} + (3/4)^{k}) \cdot \int_{\mbf{T}^k} \prod_{j=1}^{n}(1-2p + 2p\cos(2\pi  \theta^{T}w_j))\,d\theta.\]
This shows the desired inequality \cref{eq:to-show-Fourier}, setting $\gamma=e^{-\beta/4}$, and observing that then we have $\tau^{1/2} + (3/4)^{k}\le e^{-\beta k/2}+ (3/4)^{k}\le \gamma^k$ (assuming that $\beta$ was chosen to be sufficiently small and $k$ is sufficiently large).
\end{proof}

We end this section with the proof of \cref{lem:trig-iden}, and postpone the proof of \cref{lem:frieman-check} to the next section.

\begin{proof}[Proof of \cref{lem:trig-iden}]
First, note that both sides of the inequality are $\pi$-periodic, so we may assume without loss of generality that $\varphi,\varphi'\in [-\pi/2,\pi/2]$. Since $\partial^2/\partial\varphi^2 \log(\cos(\varphi)) = -\cos(\varphi)^{-2}\le 0$ for $\varphi\in (-\pi/2,\pi/2)$, the function $\log \cos(\varphi)$ is concave on $(-\pi/2,\pi/2)$ and therefore by Jensen's inequality we have
\[|\cos(\varphi)|\cdot |\cos(\varphi')|=\cos(\varphi)\cdot \cos(\varphi')\le (\cos(\varphi/2 + \varphi'/2))^2. \]
Noting that $\cos(\varphi/2 + \varphi'/2)\ge 0$ (since $\varphi/2 + \varphi'/2\in [-\pi/2,\pi/2]$), and furthermore $1-2p+2p\cos(2\varphi+2\varphi')\ge 15/16+(1/16)\cdot \cos(2\varphi+2\varphi')$ (since $p\le 1/32$), it now suffices to prove that
\[\cos(\varphi/2 + \varphi'/2)\le \frac{15}{16}+\frac{1}{16}\cdot \cos(2\varphi+2\varphi').\]
We define $x=\cos(\varphi/2 + \varphi'/2)$. Then, recalling that $\cos(2\alpha)=(\cos\alpha)^2-(\sin\alpha)^2=2(\cos\alpha)^2-1$ for all $\alpha\in \mbf R$, we have $\cos(2\varphi+2\varphi')=2(\cos(\varphi+\varphi'))^2-1=2(2x^2-1)^2-1=8x^4-8x^2+1$. Now, it suffices to check that
\[x\le x+\frac{(x-1)^2\cdot ((x+1)^2+1)}{2}=x+\frac{(x^2-2x+1)\cdot (x^2+2x+2)}{2}=\frac{x^4}{2}-\frac{x^2}{2}+1=\frac{15}{16}+
\frac{1}{16}\cdot(8x^4-8x^2+1)\]
to finish the proof of the lemma.
\end{proof}

\section{Proof of \cref{lem:frieman-check}}
We first recall an inequality due to Raikov \cite{Rai39}. This may also be derived from the Cauchy--Davenport inequality (see e.g.\ \cite[Theorem~5.4]{TVbook}) via a limiting argument.
\begin{theorem}\label{thm:kneser-mc}
Consider closed sets $A, B\subseteq \mbf{T}$. Then 
\[\mu(A+B)\ge \min(\mu(A) + \mu(B), 1).\]
\end{theorem}

We next define the \emph{compression} of a closed set $S\subseteq \mbf{T}^k$ in the $i$-th coordinate direction, for $i\in [k]$: For $(\theta_1,\ldots,\theta_k)\in \mbf T^k$, where  $\theta_1,\ldots,\theta_k\in [0,1)$, let us say that $(\theta_1,\ldots,\theta_k)\in \pi_{i}(S)$ if and only if
\[\theta_{i}\le  \int_{\mbf{T}}\mbf{1}_{(\theta_1,\ldots,\theta_{i-1},z,\theta_{i+1},\ldots,\theta_{k})\in S} \,dz.\]

We observe some properties which are immediate by construction.
\begin{fact}\label{fct:comp}
Consider any $i\in [k]$. Then for any closed set $S\subseteq \mbf{T}^k$, the set $\pi_i(S)$ is closed, and we have $\mu(\pi_i(S))=\mu(S)$. Furthermore, for any closed sets $S\subseteq S'\subseteq \mbf{T}^k$, we have $\pi_i(S) \subseteq \pi_i(S')$. 
\end{fact}
The crucial property is that $\pi_{i}(S)$ has smaller sumset than $S$. This is a continuous analogue of certain standard facts about compressions of discrete sets.
\begin{lemma}\label{lem:compression}
Consider any closed set $S\subseteq \mbf{T}^k$. We have that 
\[\mu(\pi_i(S) + \pi_i(S))\le \mu(S + S).\]
\end{lemma}
\begin{proof}
By symmetry, it suffices to prove the case when $i = k$. By \cref{fct:comp}, observe it suffices to show that $\pi_k(S)+\pi_k(S)\subseteq \pi_k(S+S)$. So, fix any $\theta,\psi\in \pi_k(S)$; we will prove that $\theta+\psi\in \pi_k(S+S)$.

First, it is convenient to introduce some notation: given a set $U\subseteq \mbf{T}^k$ and $\chi\in \mbf{T}^k$, let us define $U_{\chi} = \{z\in \mbf{T}: (\chi_1,\dots,\chi_{k-1}, z)\in U\}$. Note that $U_{\chi}$ is independent of the $k$-th coordinate of $\chi$.

Now, note that $\mu(S_\theta)\ge \theta_k$ and $\mu(S_\psi)\ge \psi_k$. By \cref{thm:kneser-mc} and the inclusion $(S+S)_{\theta+\psi}\subseteq S_\theta+S_\psi$, we have that $\mu((S+S)_{\theta+\psi})\ge \mu(S_\theta+S_\psi)\ge \min (1,\theta_k+\psi_k)$. This implies that $\theta_k+\psi_k \imod 1\in (\pi_k(S+S))_{\theta+\psi}$. Therefore $\theta+\psi\in\pi_k(S+S)$, which completes the proof.
\end{proof}

\begin{lemma}\label{lem:container}
Let $A_1,\dots,A_k\subseteq \mbf{T}$ be closed sets. Then $\pi_1 \circ \dots \circ \pi_k(A_1 \times \dots \times A_k)= [0,\mu(A_1)] \times \dots \times [0,\mu(A_k)]$.
\end{lemma}
\begin{proof}
Note that for arbitrary closed sets $B_1,\dots,B_k\subset \mbf{T}$, we have
\[\pi_\ell(B_1 \times \dots \times B_k) = B_1 \times \dots \times B_{\ell-1} \times [0,\mu(B_{\ell})] \times B_{\ell + 1} \times \dots \times B_k.\]
Whence by a simple inductive argument, we have that 
\[\pi_{\ell} \circ \cdots \circ \pi_{k}(A_1 \times \dots \times A_k) = A_1 \times \dots \times A_{\ell-1} \times [0,\mu(A_{\ell})] \times \cdots \times [0,\mu(A_{k})].\]
The desired result follows by the case $\ell = 1$.
\end{proof}

We now give the proof of \cref{lem:frieman-check}. 
\begin{proof}[{Proof of \cref{lem:frieman-check}}]
By \cref{lem:compression} and \cref{fct:comp} (both applied $k$ times), it suffices to prove that 
\[\mu(\pi_1\circ \cdots \circ \pi_k(S) + \pi_1\circ \cdots \circ \pi_k(S))\ge 2^{k}\mu(\pi_1\circ \cdots \circ \pi_k(S)).\]
Note by \cref{lem:container} and the second part of \cref{fct:comp}, we have that 
\[\pi_1 \circ \cdots \circ \pi_k(S)\subseteq [0,1/2]^{k}.\]
Therefore we may identify $\pi_1\circ \cdots \circ \pi_k(S)$ as a subset of $\mbf{R}^{k}$ (i.e., there can be no ``wrap-around'' when forming $\pi_1\circ \cdots \circ \pi_k(S)+\pi_1\circ \cdots \circ \pi_k(S)$) and conclude with the Brunn--Minkowski inequality\footnote{Formally, there is a possibility of wraparound at the boundary of $[0,1/2]^k$, but the boundary has measure zero so this causes no problems. Specifically, we may consider $(\pi_1\circ \cdots \circ \pi_k(S))\cap [0,1/2)^{k}$ and only then apply the Brunn--Minkowski inequality in $\mbf{R}^{k}$, noting that $\mu((\pi_1\circ \cdots \circ \pi_k(S))\cap [0,1/2)^{k}) = \mu(\pi_1\circ \cdots \circ \pi_k(S))$.} (see e.g.\ \cite[Theorem~3.16]{TVbook}).
\end{proof}

\bibliographystyle{amsplain0}
\bibliography{main.bib}
\end{document}